\documentclass[12pt]{amsart}
\usepackage{geometry, amssymb}
\newtheorem{theorem}{Theorem}
\newtheorem{lemma}{Lemma}
\newtheorem{proposition}{Proposition}

\def\leq{\leqslant} \def\geq{\geqslant}
\begin{document}

\title{$L^1$ means of exponential sums with multiplicative coefficients. I.}
\author{Mayank Pandey}
\address{Department of Mathematics, Princeton University, Princeton, NJ 08540, USA}
\email{mayankpandey9973@gmail.com}

\author{Maksym Radziwill}
\address{UT Austin, Department of Mathematics, PMA 8.100, 2515 Speedway, Stop C1200, Austin, TX 78712}
\email{maksym.radziwill@gmail.com}

\begin{abstract}
  We show that the $L^1$ norm of an exponential sum of length $X$ and with coefficients equal to the Liouville or M\"{o}bius function is at least $\gg_{\varepsilon} X^{1/4 - \varepsilon}$ for any given $\varepsilon$. 

  For the Liouville function this improves on the lower bound $\gg X^{c/\log\log X}$ due to Balog and Perelli (1998). For the M\"{o}bius function this improves the lower bound $\gg X^{1/6}$ due to Balog and Ruzsa (2001). The large discrepancy between these lower bounds is due to the method employed by Balog and Ruzsa, as it crucially relies on the vanishing of $\mu(n)$. Instead our proof puts the two cases on an equal footing by exploiting the connection of these coefficients with zeros of Dirichlet $L$-functions. In the second paper in this series we will obtain a lower bound $\gg X^{\delta}$ for some small $\delta$ but for general (non-pretentious) multiplicative functions. 
  
  %The discrepancy between the last two lower bounds was due to the fact that the proof of Balog and Ruzsa relies crucially on the vanishing of M\"{o}bius on non-squarefree integers and is thus not applicable to the Liouville function. Instead, our proof relies on the connection of the Liouville and M\"{o}bius function with the Riemann $\zeta$-function, thus putting these two cases on an equal footing. 
\end{abstract}

\maketitle

\section{Introduction}

The behavior of $L^1$ norms of exponential sums, i.e.
\begin{equation} \label{eq:norm}
\int_{0}^{1} \Big | \sum_{n \leq X} a(n) e(n \alpha) \Big | d \alpha \ , \ \sum_{n \leq X} |a(n)|^2 \asymp X
\end{equation}
is a classical topic in harmonic analysis. It follows from works of Konyagin~\cite{littlewood2} and McGehee-Pigno-Smith~\cite{littlewood}, resolving a conjecture of Littlewood, that~\eqref{eq:norm} is always at least 
$\log X$ if $a(n)$ takes on a fixed set of values. This lower bound is attained for $a(n)$ the indicator function of sequences of an additive nature, for example the indicator function of an arithmetic progression. 

For multiplicative sequences such as $\mu(n)$ or $\lambda(n)$, we expect the lower bound to be closer to $\sqrt{X}$ but known results fall short of this expectation. For $a(n) = \lambda(n)$ and $a(n) = \mu(n)$ a lower bound of respectively $X^{c/\log\log X}$ and $X^{1/6}$ is due to respectively Balog-Perelli~\cite{BalogPerelli} and Balog-Ruzsa~\cite{BalogRuzsa2}. The large discrepancy between these bounds arises because the proof of Balog-Ruzsa crucially uses that $\mu(p^2 m) = 0$ for integers $m$ and a sequence of primes $p$, while $\lambda(n) \neq 0$ for all integers $n$. 

Our main aim in this paper is to put both of these cases on the same footing and obtain improved lower bounds in both cases. Since both cases are similar we will present a full proof only for $a(n) = \lambda(n)$, the hitherto harder case. 

\begin{theorem}\label{thm:main}
  For every $\varepsilon > 0$, 
  \begin{equation} \label{eq:mainnorm}
  \int_{0}^{1} \Big | \sum_{n \leq X} \lambda(n) e(n \alpha) \Big | d \alpha \gg_{\varepsilon} X^{1/4 - \varepsilon}.
  \end{equation}
  \end{theorem}

The main arithmetic input in our proof is the observation that for any primitive character $\chi$, the Liouville function resonates with $\chi(n) n^{-i \gamma}$ provided that $\beta + i \gamma$ is a zero of $L(s, \chi)$. Ultimately our lower bound comes from the fact that there exists a Dirichlet $L$-function with a zero with real part $\geq \tfrac 12$, this readily follows for each $L(s, \chi)$ from the functional equation and the explicit formula\footnote{Indeed if $L(s, \chi)$ had no zeros in the critical strip then the explicit formula would imply that the sum of $\chi(p)$ with $p \leq x$ is a smooth function of $x$, a contradiction. Once a zero exists the functional equation implies that there is a zero with real part $\geq \tfrac 12$.}.

We now state this more precise version of Theorem~\ref{thm:main}.

\begin{theorem}\label{thm:extended}
  Let $\chi$ be a primitive character.
  Let $\rho_{\chi} = \beta_{\chi} + i \gamma_{\chi}$ denote a zero of $L(s, \chi)$. 
  Then, for every $\varepsilon > 0$,
  \[
      \int_{0}^{1} \Big | \sum_{n \leq X} \lambda(n) e(n \alpha) \Big | d \alpha \gg_{\rho, \varepsilon, \chi} X^{\beta_{\chi} / 2 - \varepsilon}. 
  \]
\end{theorem}

%With a small additional effort Theorem~\ref{thm:extended} can be extended to also cover the case of the trivial character $\chi_0 \equiv 1$, i.e $L(s, \chi_0) = \zeta(s)$. 
Theorem~\ref{thm:extended} has some amusing implications. First, if Siegel zeros exist (in the very weak sense that $\beta_{\chi} \rightarrow 1$ for a subsequence of characters $\chi$), then the lower bound in Theorem~\ref{thm:main} can be improved to $X^{1/2 - \varepsilon}$. Second, an unrealistic (but not currently ruled out!) upper bound for \eqref{eq:mainnorm} of the form $X^{1/2 - \delta}$ with $\delta > 0$ implies a zero-free strip for every Dirichlet $L$-function. As a final consequence of Theorem~\ref{thm:extended}, if one aims to just slightly improve Theorem~\ref{thm:main}, one can assume without loss of generality the Riemann Hypothesis for every $L(s, \chi)$. 

Currently, even conditionally on the Generalized Riemann Hypothesis (GRH) we do not know how to improve Theorem \ref{thm:main}. We describe below a simple argument that recovers Theorem \ref{thm:main} conditionally on GRH. It follows from the work of Baker and Harman \cite{BakerHarman} that for all $\varepsilon > 0$, 
%We remark that under the Generalized Riemann Hypothesis, by approximating $\alpha$ by a rational and opening into characters (this was executed for the M\"obius function where the situation is the same by 
%Baker and Harman~\cite{BakerHarman}), one may show% Baker and Harman~\cite{BakerHarman} showed that uniformly in $\alpha$, we have
\begin{equation} \label{eq:baker}
  |S(\alpha; X)| \ll_{\varepsilon} X^{3/4 + \varepsilon} \ , \ S(\alpha; X) := \sum_{n \leq X} \lambda(n) e(n \alpha). 
\end{equation}
We then notice that
$$
X = \int_{0}^{1} |S(\alpha; X)|^2 d \alpha \ll X^{3/4 + \varepsilon} \int_{0}^{1} |S(\alpha; X)| d \alpha,
$$
and thus Theorem \ref{thm:main} follows conditionally on GRH. 
Improving Baker and Harman's bound \eqref{eq:baker} presents significant challenges. Any method based on Type-II sums is a no-go: when the bilinear form factors into two terms of equal length we can save at most $X^{1/4}$ over the trivial bound. Thus we need to exploit in that case cancellations in a sum over low-lying zeros of different $L$-functions. This is however currently way out of reach. 

Before proceeding to the proof we give a quick survey of related results. When $a(n) = \Lambda(n)$, Vaughan~\cite{Vaughan} shows that \eqref{eq:norm} is at least $\gg\sqrt{X}$ and at most $\ll\sqrt{X \log X}$. The lower bound is non-trivial and comes from the fact that the behavior of the exponential sum near the major arcs can be understood; one can hope for similar approaches to work whenever the sequence $a(n)$ is non-negative.  For sequences $a(n)$ such as the divisor function or Fourier coefficients of a cusp form, an asymptotic of size $\sqrt{X}$ was recently obtained by the first author~\cite{Pandey}. The main idea in that case is to use a recursive mechanism based on Voronoi's summation formula; this fundamentally uses the automorphic provenance of the coefficients. Note that in the case of cusp form Fourier coefficients, positivity is not available and it is the existence of a nice summation formula that is crucial. Finally for $a(n) = \mu_r(n)$ the indicator function of $r$-free numbers, Balog and Ruzsa~\cite{BalogRuzsa2} determined the order of magnitude of \eqref{eq:norm} to be $X^{1/(r + 1)}$ improving earlier work of Br\"udern, Granville, Perelli, Vaughan, and Wooley~\cite{etal}. Recently a short interval generalization was considered in~\cite{yuchen}. Finally Balog and Ruzsa~\cite{BalogRuzsa1} have also obtained a lower bound for \eqref{eq:norm} valid for any sequence that vanishes on non-squarefree integers. This may be substantially improved as a corollary of work of Konyagin~\cite{Konyagin}. A similar result for sequences supported on primes was recently obtained in~\cite{generalization}.

We mention also that in the next paper in this series we will obtain results valid for general multiplicative functions using different methods. 

%Specifically for non-negative sequences $a(n)$ rather satisfactory results can be obtained. For example for $a(n) = \Lambda(n)$ Vaughan obtains a lower bound of size $\sqrt{X}$ using information about the behavior near major arcs, while for $a(n) = d(n)$ Pandey obtains an asymptotic of size $\sqrt{X}$ using a recursive mechanism based on Voronoi's summation formula. Finally for $a(n) = \mu_r(n)$ the indicator function of $r$-free numbers the $L^1$ norm is of size $X^{1/(r + 1)}$ as shown by Balog and Ruzsa, improving earlier work of Br\"{u}dern, Granville, Perelli, Vaughan and Wooley. The case of short intervals was recently investigated by \cite{KaisaStudent}. In fact in \cite{...} Balog and Ruzsa established a general lower bound for sequences that vanish on non-squarefree integers, and this was slightly strenghtened recently in \cite{...}.

\subsection{Notation}
We write $e(x) := e^{2\pi i x}$. Given a smooth function $f$ we define respectively its Fourier transform, Mellin transform, and Sobolev norm as
\[
    \widehat{f}(u) := \int_{\mathbb{R}} f(x) e( - x u) dx \ , \ \widetilde{f}(s) := \int_{0}^{\infty} f(x) x^{s - 1} dx \ , \
    \| f \|_{p, r} := \sum_{i = 0}^{r} \Big ( \int_{\mathbb{R}} |f^{(i)}(x)|^p \Big )^{1/p}. 
\]

\section{Proof of Theorem \ref{thm:main} and Theorem \ref{thm:extended}}

Theorem \ref{thm:main} follows immediately from Theorem \ref{thm:extended}: Pick $\chi$ to be the trivial character and notice that $L(s, \chi) = \zeta(s)$ has a zero on the half-line with imaginary part $\gamma \approx 14.134725 \ldots$. In turn Theorem \ref{thm:extended} follows from two main propositions that we state below.

\begin{proposition} \label{prop:A}
  Let $\chi$ be a primitive character of conductor $q$.  
  Let $w$ and $\phi$ be two smooth functions, compactly supported in $(0,1)$. Let $t$ be a real number.
  Let $a(n)$ be a sequence of real numbers. For any prime $p$, define a new sequence $a_p(n) := a(n) \cdot (p \mathbf{1}_{p | n} - 1)$. Then,
  \begin{align*}
      \| \phi \|_{1,2} \|w \|_{1,0} \cdot q^2(1 + t^2) \int_{0}^{1} \Big | \sum_{n \leq X} a(n) & e(n \alpha) \Big | d \alpha \\ & \gg  \frac{1}{X} \sum_{p \leq \sqrt{X}} \Big | \int_{p^2}^{\infty} w \Big ( \frac{y}{X} \Big ) \mathcal{S}(a_{p}, t, \chi,  \phi)(y) \cdot \frac{dy}{y} \Big |,
  \end{align*}
  where we define for any sequence $a(n)$
  \begin{equation} \label{eq:define}
      \mathcal{S}(a, t, \chi, \phi)(y) := \sum_{n} a(n) \chi(n) n^{-it} \phi \Big ( \frac{n}{y} \Big ).
    \end{equation}
\end{proposition}

In the next Proposition we establish the existence of a special kernel that resonates cleanly with the partial sums of $\lambda(n) \chi(n) n^{-i \gamma}$ when $\gamma$ is the ordinate of a zero of $L(s, \chi)$. The construction of such kernels can be traced back to Pintz \cite{Pintz}. 

\begin{proposition} \label{prop:B} 
  Let $\chi$ be a primitive Dirichlet character.
  Let $\rho = \beta + i \gamma$ denote a zero of $L(s, \chi)$.
  Let $A > 10$. 
  There exists a smooth function $w_{A, \rho,\chi}$ compactly supported in $(0,1)$ such that
  \begin{enumerate}
  \item for all $0 \leq u \leq 1$ we have $w_{A, \rho, \chi}(u) \ll_{\rho, \chi} u^{A}$.
  \item we have, for any $Y > 1$, 
      \begin{align} \label{eq:stuff}
          \int_{0}^{\infty} w_{A, \rho, \chi} \Big ( \frac{y}{Y} \Big ) \mathcal{S}(\lambda, \gamma, \chi, \phi )(y) \cdot \frac{dy}{y} = \widetilde{\phi} ( \beta ) Y^{\beta} + O_{A, \rho,\chi}(Y^{-A})
      \end{align}
   with $\mathcal{S}(a, t, \chi, \phi)$ defined in \eqref{eq:define}.
  \end{enumerate}
\end{proposition}
%If one were to include the case of the trivial character $\chi_0$ there would be an additional term on the right-hand side of \eqref{eq:stuff} arising from the pole of $L(s, \chi_0) = \zeta(s)$. 

We are now ready to conclude the proof.
\begin{proof}[Proof of Theorem \ref{thm:extended}]
  Let $\rho = \beta + i \gamma$ be a zero of $L(s, \chi)$. 
  Let $0 \leq \phi \leq 1$ be a smooth function compactly supported in $(0, 1)$. The function $\phi$ is fixed once and for all thus the implicit constants in $\gg$ will not depend on $\phi$. 
  First notice that for each $p \leq X^{1/2 - \varepsilon}$, due to the rapid decay of the kernel $w_{A, \rho, \chi}$ (where we pick $A = 100 / \varepsilon$) we have
  $$
  \int_{p^2}^{\infty} w_{A, \rho, \chi} \Big ( \frac{y}{X} \Big ) \mathcal{S}(\lambda_{p}, t, \chi, \phi)(y) \cdot \frac{dy}{y} = \int_{0}^{\infty} w_{A, \rho, \chi} \Big ( \frac{y}{X} \Big ) \mathcal{S}(\lambda_{p}, t, \chi, \phi)(y) \cdot \frac{dy}{y} + O_{\rho, \varepsilon, \chi}(X^{-10}).
  $$
  Since $\lambda$ is completely multiplicative, we have
  $$
  \mathcal{S}(\lambda_{p}, \gamma, \chi, \phi)(y) = - p^{1-i\gamma} \chi(p) \mathcal{S}(\lambda, \gamma, \chi, \phi)(y / p) - \mathcal{S}(\lambda, \gamma, \chi, \phi)(y).
  $$
  Therefore, by Proposition \ref{prop:B}, we have
  \begin{align*}
      \int_{0}^{\infty} w_{A, \rho, \chi} \Big ( \frac{y}{X} \Big ) & \mathcal{S}(\lambda_{p}, t, \chi, \phi)(y) \cdot \frac{dy}{y} \\ & = - \widetilde{\phi} ( \beta) \cdot \Big ( p^{-i \gamma} \chi(p) p \cdot \Big ( \frac{X}{p} \Big )^{\beta} + X^{\beta} \Big ) + O_{\rho, \varepsilon, \chi}(X^{-10}). 
  \end{align*}
  In absolute value this is $\geq \widetilde{\phi}(\beta) \cdot (X^{\beta} p^{1 - \beta} - X^{\beta}) \gg X^{\beta} p^{1 - \beta}$ for $p$ co-prime to the conductor of $\chi$. 
  Therefore by Proposition \ref{prop:A},
  $$
  \int_{0}^{1} \Big | \sum_{n \leq X} \lambda(n) e(n \alpha) \Big | d \alpha \gg_{\rho, \varepsilon, \chi} \frac{1}{X} \sum_{\substack{p \leq X^{1/2 - \varepsilon}\\ p\nmid q}} X^{\beta} p^{1 - \beta} \gg X^{\beta / 2 - \varepsilon},$$
  and the result follows.
\end{proof}

\section{Proof of Proposition \ref{prop:A}}

We first need a bound for $L^1$ norms of exponential sums with coefficients composed of multiplicative characters. We note that the exponents in the Lemma below can be sharpened but such sharper versions are of no use to us. 
\begin{lemma} \label{le:voronoi}
  Let $\chi$ be a primitive character of conductor $q$. 
  Let $0 \leq \phi \leq 1$ be
  a smooth function compactly supported in $(0, 1)$. Then, for all $ y > 1$, real $t$,
  $$
  \int_{0}^{1} \Big | \sum_{n} n^{-it} \chi(n) \phi \Big ( \frac{n}{y} \Big ) e(n \alpha) \Big | dt \ll (1 + t^2) q^2 \cdot \| \phi \|_{1,2}.
  $$
\end{lemma}

\begin{proof}
  Notice that it's enough to bound
  $$
  \int_{-1/2}^{1/2} \Big | \sum_{n} n^{-it} \chi(n) \phi \Big ( \frac{n}{y} \Big )  e(n \alpha) \Big |. 
  $$
  Furthermore for every integer $|a| \leq q/2$ the contribution of all the neighborhoods $|\alpha - a / q| \leq 1 / (q y)$ is trivially bounded by $\ll 1$. 
  We thus assume $1/(q y) < |\alpha - a / q|$ for every integer $|a| \leq q / 2$ and $|\alpha| \leq \tfrac 12$.  
  Poisson summation gives
  \begin{align} \label{eq:sum}
  \sum_{n} n^{-it} & \chi(n) \phi \Big ( \frac{n}{y} \Big ) e(n \alpha) \\ \nonumber & = \frac{1}{q} \sum_{\ell} \Big ( \sum_{x \pmod{q}} \chi(x) e \Big ( - \frac{x \ell}{q} \Big ) \Big ) \int_{0}^{\infty} x^{-it} \phi \Big ( \frac{x}{y} \Big ) e \Big (x \alpha - \frac{\ell x}{q} \Big ) d x. 
  \end{align}
  By a change of variables $u := x / y$ the integral is equal to 
  $$y^{1 - it} \int_{0}^{\infty} \phi (u) u^{-it} e \Big (u y \Big (  \alpha - \frac{\ell}{q} \Big ) \Big ) d u.$$
  Integrating by parts twice, we see that this is 
  $$
  \ll y (1 + |t|)^{2} \cdot (y |\alpha - \ell / q|)^{-2} \cdot \| \phi \|_{1,2}.
  $$
  Therefore \eqref{eq:sum} is, for $|\alpha| \leq \tfrac 12$,
  $$
  \ll \| \phi \|_{1,2} \cdot \Big ( \frac{1 + t^2}{y} \sum_{|\ell| \leq q / 2} |\alpha - \ell / q|^{-2} \Big).% + (t q)^2 \cdot \frac{1}{y} \Big ).
  $$
  Integrating this over $\alpha$ with $1 / (q y) \leq |\alpha - \ell / q|$ for every $|\ell| \leq q / 2$ then yields the final bound $\ll (1 + t^2) q^2 \| \phi \|_{1,2}$. 
\end{proof}

\begin{lemma} \label{le:young}
  Let $\chi$ be a primitive character of conductor $q$. 
  Let $0 \leq \phi \leq 1$ be smooth and compactly supported in $(0, 1)$. Let $t$ be a real number.
  Then, for any $y \leq X$, 
  $$
  \| \phi \|_{1,2} \cdot q^2(1 + t^2) \int_{0}^{1} \Big | \sum_{n \leq X} \lambda(n) e(n \alpha) \Big | d \alpha \gg \int_{0}^{1} \Big | \sum_{n} \lambda(n) \chi(n) n^{-it} \phi \Big ( \frac{n}{y} \Big ) e(n \alpha) \Big | d \alpha.
  $$
\end{lemma}
\begin{proof}
  Notice that
  \begin{align*}
  \sum_{n} \lambda(n) \chi(n) & n^{-it} \phi \Big ( \frac{n}{y} \Big ) e(n \alpha) \\ & = \int_{0}^{1} \Big ( \sum_{n \leq X} \lambda(n) e(n (\alpha - \beta)) \Big ) \cdot \Big ( \sum_{n} \chi(n) n^{-it} \phi \Big ( \frac{n}{y} \Big ) e(n \beta) \Big ) d \beta. 
  \end{align*}
  Therefore, by Young's inequality, 
  \begin{align*}
  \int_{0}^{1} \Big | \sum_{n} \lambda(n) \chi(n) & n^{-it} \phi \Big ( \frac{n}{y} \Big ) e(n \alpha) \Big | d \alpha \\ & \leq \int_{0}^{1} \Big | \sum_{n \leq X} \lambda(n) e(n \alpha) \Big | d \alpha \int_{0}^{1} \Big | \sum_{n} \chi(n) n^{-it} \phi \Big ( \frac{n}{y} \Big ) e(n \beta) \Big | d \beta. 
  \end{align*}
The result now follows from Lemma \ref{le:voronoi}. 
\end{proof}

\begin{lemma}
  There exists a smooth $0 \leq \varrho \leq 1$, compactly supported in $(-1/2, 1/2)$ and such that $\widehat{\varrho}(x) > 0$ for all $x \in \mathbb{R}$. 
\end{lemma}
\begin{proof}
  Let $f(x) := 1 - |2 x|$ for $|x| \leq 1/2$ and $0$ otherwise. As is well-known the Fourier transform of $f$ is non-negative. Consider now
  $$\varrho(x) := f(x) \widehat{f}(-x).$$
  This is non-negative, compactly supported in $(-1/2,1/2)$ and bounded above by $\widehat{f}(0) = 1/2$. Finally,
  $$
  \widehat{\varrho}(u) = \int_{\mathbb{R}} \widehat{f}(u - x) f(x) dx.
  $$
  Both $\widehat{f}$ and $f$ are non-negative, therefore the above convolution is non-negative, in fact strictly positive, since the only way it could be zero is if $\widehat{f}$ were zero on an interval, but this is impossible because $\widehat{f}$ is an entire function.  
\end{proof}

With these lemmas at hand we are ready to prove Proposition \ref{prop:A}.

\begin{proof}[Proof of Proposition \ref{prop:A}]

  Fix $0 \leq \varrho \leq 1$, a smooth function compactly supported in $(-1/2,1/2)$ and with $\widehat{\varrho} > 0$. Let
  $$
  \phi_{1}(x) := \frac{1}{m} \cdot \frac{\phi(x)}{\widehat{\varrho}(x)} \ , \ m := \min_{x \in (0,1)} \widehat{\varrho}(x),
$$
so that $0 \leq \phi_1 \leq 1$ and $\phi_1$ is compactly supported and smooth (note that $\widehat{\varrho}$ is entire because $\varrho$ is compactly supported). By Lemma \ref{le:young}, for any $y \leq X$ we have
$$
\| \phi_1 \|_{1,2} q^2(1 + t^2) \cdot \int_{0}^{1} \Big | \sum_{n \leq X} a(n) e(n \alpha) \Big | d \alpha \gg \int_{0}^{1} \Big | \sum_{n} a(n) \chi(n) n^{-it} \phi_1 \Big ( \frac{n}{y} \Big ) e(n \alpha) \Big | d \alpha. 
$$
Furthermore, we can isolate in the above integral neighborhoods of
$$
\Big ( \frac{b}{p} - \frac{1}{2 y}, \frac{b}{p} + \frac{1}{2 y} \Big )
$$
with $(b,p) = 1$ and $p \leq \min(\sqrt{y}, \sqrt{X})$. We notice that these are disjoint. 
Therefore, 
$$
\int_{0}^{1} \Big | \sum_{n} a(n) \phi_1 \Big ( \frac{n}{y} \Big ) e(n \alpha) \Big | d \alpha \geq \sum_{\substack{(b,p) = 1 \\ p \leq \min(\sqrt{y}, \sqrt{X})}} \int_{\mathbb{R}} \varrho(y \theta) \Big | \sum_{n} a(n) \chi(n) n^{-it} \phi_1 \Big ( \frac{n}{y} \Big ) e \Big ( \frac{n b}{p} + n \theta \Big ) \Big | d \theta. 
$$
In turn the above is greater than
\begin{align*}
    \sum_{p \leq \min(\sqrt{y}, \sqrt{X})} & \Big | \sum_{n} a(n) \chi(n) n^{-it} \Big ( \sum_{(b,p) = 1} e \Big ( \frac{n b}{p} \Big ) \Big ) \phi_1 \Big ( \frac{n}{y} \Big ) \int_{\mathbb{R}} \varrho(\theta y) e(n \theta) d \theta \Big | \\ & = \sum_{p \leq \min(\sqrt{y}, \sqrt{X})} \Big | \sum_{n} a(n) \chi(n) n^{-it} \cdot \Big ( p \mathbf{1}_{p | n} - 1 \Big ) \phi_1 \Big ( \frac{n}{y} \Big ) \cdot \frac{1}{y} \cdot \widehat{\varrho} \Big ( \frac{n}{y} \Big ) \Big | .
\end{align*}
We recall that by definition $\phi_{1} \widehat{\varrho} = \frac{1}{m} \phi$, therefore we have obtained the lower bound
\begin{align*}
    \| \phi_1 \|_{1,2} \cdot q^2(1 + t^2) \int_{0}^{1} \Big | \sum_{n \leq X} a(n) e(n \alpha) \Big | d \alpha \gg \sum_{p \leq \min(\sqrt{y}, \sqrt{X})} \Big | \mathcal{S}(a_p, t, \chi, \phi)(y) \cdot \frac{1}{y} \Big | .
\end{align*}
We now integrate both sides with respect to $|w|(y / X)$, this yields
\begin{align*}
X \| w \|_{1, 0} \| \phi_1 \|_{1,2} \cdot q^2(1 + t^2) \int_{0}^{1} \Big | \sum_{n \leq X} a(n) & \chi(n) n^{-it} e(n \alpha) \Big | d \alpha \\ & \gg \sum_{p \leq \sqrt{X}} \Big | \int_{p^2}^{\infty} w \Big ( \frac{y}{X} \Big ) \mathcal{S}(a_p, t, \chi, \phi) \cdot \frac{dy}{y} \Big |.
\end{align*}
It remains to notice that $\| \phi_1 \|_{1,2} \ll \| \phi \|_{1,2}$ since $\widehat{\varrho}$ is fixed once and for all. 

\end{proof}

\section{Proof of Proposition \ref{prop:B}}

Let $\gamma$ be the ordinate of a zero of $L(s, \chi)$ with real-part $\beta \geq \tfrac 12$. 
Consider now the kernel
$$
w_{A, \rho, \chi}(u) := \frac{1}{2\pi i} \int_{(2)} \frac{(s + i\gamma - 1)L(s + i \gamma, \chi)}{L(2 \beta + 2 i \gamma, \chi) (\beta + i\gamma - 1) \cdot (s - \beta)} \cdot u^{-s} \cdot \Big ( \frac{1 - e^{-{(s - \beta)}}}{s - \beta} \Big )^{2A} ds,
%e^{(s - \beta)^2 } ds   \ , \ u > 0 .
$$
where we note that $L(2 \beta +2 i \gamma)(\beta + i \gamma - 1) \neq 0$ since $\beta \geq \tfrac 12$ and because there are no zeros with $\beta = 1$.  
We summarize its properties in the lemma below.
\begin{lemma} \label{le:w}
  Let $\chi$ be a primitive character.
  Let $\rho = \beta + i \gamma$ be a zero of $L(s, \chi)$ with $\gamma > 0$.
  Let $A \geq 10$. 
  The function $w_{A, \rho, \chi} : [0, \infty) \mapsto \mathbb{C}$ is smooth and compactly supported in $(0,1)$ such that
    $$
    w_{A, \rho, \chi}(u) \ll_{A, \rho, \chi} u^{A}
    $$
    for all $0 \leq u \leq 1$.
    Finally, suppose $0 \leq \phi \leq 1$ is a smooth function compactly supported in $(0, 1)$, then, for all $Y$, and for any $A > 10$, 
  \begin{align*}
      \int_{0}^{\infty} w_{A, \rho, \chi} \Big ( \frac{y}{Y} \Big ) \cdot \Big ( \sum_{n} \lambda(n) \chi(n) n^{- i \gamma} \phi \Big ( \frac{n}{y} \Big ) \Big ) \frac{d y}{y} = \widetilde{\phi} (\beta)  Y^{\beta} + O_{A, \rho, \chi}(Y^{-A})
  \end{align*}
\end{lemma}
\begin{proof}
  For the first claim, notice that we have
  $$
  w_{A, \rho, \chi}(s) =  \frac{1}{C} \sum_{n \geq 1} \chi(n) n^{-i \gamma} \cdot \frac{1}{2\pi i} \int_{(2)} \frac{s + i\gamma - 1}{s - \beta} \cdot \Big ( \frac{1 - e^{-(s - \beta)}}{s - \beta} \Big )^{2A} \cdot (n u)^{-s} du, 
  $$
  with $C := L(2 \beta + 2 i \gamma, \chi)(\beta + i\gamma - 1)$. If $u > 1$, then for every $n \geq 1$ upon shifting the contour to the right we see that
  $$
  \frac{1}{2\pi i} \int_{(2)} \frac{s + i\gamma - 1}{s - \beta} \cdot \Big ( \frac{1 - e^{-(s - \beta)}}{s - \beta} \Big )^{2A} \cdot (n u)^{-s} ds = 0.  
  $$
  On the other hand, for the bound for $0 \leq u \leq 1$, we shift the contour to the left to the line $\Re s = -A$, noticing that the function,
  $$
  \frac{(s + i\gamma - 1)L(s + i \gamma, \chi)}{s - \beta}
  $$
  is entire and grows no faster than $\ll (q (\gamma + |s|)^{\tfrac 12 - \Re s}$.
 For the last claim, we notice that by a change of variable, 
  \begin{align*}
      \int_{0}^{\infty} w_{A, \rho, \chi} \Big ( \frac{y}{Y} \Big ) \cdot \Big ( & \sum_{n} \lambda(n) \chi(n) n^{- i \gamma} \phi \Big ( \frac{n}{y} \Big ) \Big ) \frac{d y}{y} \\ & = \int_{0}^{\infty} \phi \Big ( \frac{1}{y} \Big ) \Big ( \sum_{n} \lambda(n) \chi(n) n^{-i \gamma} w_{A, \rho, \chi} \Big ( \frac{n y}{Y} \Big ) \Big ) \frac{d y}{y}.
  \end{align*}
  Opening up the definition of $w_{A, \rho, \chi}$ and using the fact that
  $$
  \sum_{n \geq 1} \frac{\lambda(n) \chi(n)}{n^{s}} = \frac{L(2s, \chi)}{L(s, \chi)},
  $$
  we get
  \begin{align*}
      \sum_{n} & \lambda(n) \chi(n) n^{-i \gamma} w_{A, \rho, \chi} \Big ( \frac{n y}{Y} \Big ) \\ & = \frac{1}{2\pi i} \int_{(2)} \frac{Y^s}{y^s} \cdot \frac{(s + i\gamma - 1)L(2 s + 2 i \gamma, \chi)}{L(2 \beta + 2 i \gamma, \chi)(\beta + i\gamma - 1) \cdot (s - \beta)} \cdot \Big ( \frac{1 - e^{-(s - \beta)}}{s - \beta}  \Big )^{2A} ds \\ & = \Big ( \frac{Y}{y} \Big )^{\beta} + O_{A, \rho, \chi}(Y^{-A}),
  \end{align*}
  again by shifting to the line $\Re s = - A$. 
  Integrating the above with respect to $\phi(1/y) d y / y$ yields the claim.
\end{proof}

\bibliography{l1}

\begin{thebibliography}{10}

\bibitem{BakerHarman}
R.~C. Baker and G.~Harman.
\newblock Exponential sums formed with the {M}öbius function.
\newblock {\em Journal of the London Mathematical Society}, s2-43(2):193--198,
  1991.

\bibitem{BalogPerelli}
A.~Balog and A.~Perelli.
\newblock On the {$L^1$} {M}ean of the {E}xponential {S}um {F}ormed with the
  {M}{\"o}bius {F}unction.
\newblock {\em Journal of the London Mathematical Society}, 57(2):275--288,
  1998.

\bibitem{BalogRuzsa1}
A.~Balog and I.~Z. Ruzsa.
\newblock A {N}ew {L}ower {B}ound for the {$L^1$} {M}ean of the {E}xponential
  {S}um with the {M}{\"o}bius {F}unction.
\newblock {\em Bulletin of the London Mathematical Society}, 31(4):415--418,
  1999.

\bibitem{BalogRuzsa2}
A.~Balog and I.~Z. Ruzsa.
\newblock On the {E}xponential sum over {$r$}-{F}ree {I}ntegers.
\newblock {\em Acta Mathematica Hungarica}, 90:219--230, 2001.

\bibitem{etal}
J.~Br{\"u}dern, A.~Granville, A.~Perelli, R.~C. Vaughan, and T.~D. Wooley.
\newblock On the exponential sum over {$k$}-free numbers.
\newblock {\em R. Soc. Lond. Philos. Trans. Ser. A Math. Phys. Eng. Sci.},
  356(1738):739--761, 1998.

\bibitem{generalization}
E.~Eckels, S.~Jin, A.~H. Ledoan, and B.~Tobin.
\newblock Linnik's large sieve and the {$L^{1}$} norm of exponential sums.
\newblock {\em Bulletin of the London Mathematical Society}, 55(2):843--853,
  2023.

\bibitem{littlewood2}
S.~V. Konyagin.
\newblock On the {L}ittlewood problem.
\newblock {\em Izv. Akad. Nauk SSSR Ser. Mat.}, 45(2):243--265, 1981.

\bibitem{Konyagin}
S.~V. Konyagin.
\newblock Problems on the set of squarefree numbers.
\newblock {\em Izvestiya: Mathematics}, 68(3):493--520, 2004.

\bibitem{littlewood}
O.~C. McGehee, L.~Pigno, and B.~Smith.
\newblock Hardy’s inequality and the {$L^1$} norm of exponential sums.
\newblock {\em Ann. of Math.}, 113:613--618, 1981.

\bibitem{Pandey}
M.~Pandey.
\newblock On the distribution of additive twists of the divisor function and
  {H}ecke eigenvalues.
\newblock {\em arxiv:2110.03202}, 2021.

\bibitem{Pintz}
J.~Pintz.
\newblock Oscillatory properties of {$M(x) = \sum_{n \leq x} \mu(n)$}, {I}.
\newblock {\em Acta Arithmetica}, 42(1):49–55, 1983.

\bibitem{yuchen}
Y.~C. Shen.
\newblock On the {B}alog-{R}uzsa theorem in short intervals.
\newblock {\em arxiv: 2204.13531}, 2022.

\bibitem{Vaughan}
R.~C. Vaughan.
\newblock {The {$L^1$} mean of {E}xponential {S}ums over {P}rimes}.
\newblock {\em Bulletin of the London Mathematical Society}, 20(2):121--123, 03
  1988.

\end{thebibliography}
\bibliographystyle{plain}

\end{document}